\documentclass{article}
\usepackage{adjustbox}
\usepackage[a4paper, left=3.0cm,right=3.0cm,top=3.85cm,bottom=3.85cm]{geometry}
\usepackage{amsmath,amsthm,amssymb}
\usepackage{graphicx} 
\usepackage{subcaption}
\DeclareMathOperator*{\argmin}{arg min}
\DeclareMathOperator{\Span}{span}
\usepackage[numbers]{natbib} 
\usepackage[hidelinks]{hyperref}     
\usepackage{doi}
\usepackage{authblk}
\usepackage{xcolor}
\usepackage{hyperref}
\usepackage{cleveref}
\date{}
\title{Symplecticity-Preserving Prediction of parameter-dependent Hamiltonian Dynamics by Generalized Kernel Interpolation}
\author[1]{Robin Herkert \thanks{\texttt{robin.herkert@ians.uni-stuttgart.de}}}
\author[2]{Tobias Ehring \thanks{\texttt{tobias.ehring@ians.uni-stuttgart.de}}}
\author[3]{Bernard Haasdonk\thanks{\texttt{haasdonk@ians.uni-stuttgart.de}}}
\affil[1,2,3]{Institute of Applied Analysis and Numerical Simulation, University of Stuttgart, Pfaffenwaldring 57,
70569 Stuttgart, Baden-Württemberg, Germany}

\begin{document}
\maketitle
\newcommand{\p}{\mathbf{p}}
\newcommand{\q}{\mathbf{q}}
\renewcommand{\P}{\mathbf{P}}
\newcommand{\Q}{\mathbf{Q}}
\newcommand{\A}{\mathbf{A}}
\newcommand{\B}{\mathbf{B}}
\newcommand{\C}{\mathbf{C}}
\newcommand{\D}{\mathbf{D}}

\newcommand{\Ham}{\mathcal{H}}
\newcommand{\R}{\mathbb{R}}
\newcommand{\N}{\mathbb{N}}
\newcommand{\openQ}[1]{{\color{red} Q: #1}}
\newtheorem{theorem}{Theorem}
\newtheorem{lemma}{Lemma}
\newtheorem{corollary}{Corollary}
\newtheorem{definition}{Definition}
\newtheorem{remark}{Remark}
\newtheorem{proposition}{Proposition}
\newcommand{\Hk}{H_k(\Omega)}
\newcommand{\ip}[2]{\langle #1,#2\rangle_{\Hk}}
\newcommand{\norm}[1]{\left\lVert #1\right\rVert_{\Hk}}

\begin{abstract}
We extend the kernel-based symplectic predictor of \cite{herkert2026kernel}
to a parameter-augmented setting in which the learned flow-map surrogate
depends not only on the state, but also on additional variables such as
physical parameters and macro time-step sizes. The method uses a product
kernel ansatz on a parameter and macro step augmented domain and constructs the prediction through an
implicit symplectic-Euler-type update. Hence, for every fixed admissible
parameter and time-step instance, the resulting large-step predictor is
symplectic by construction. The training problem is formulated as gradient
Hermite--Birkhoff interpolation in a reproducing kernel Hilbert space. Efficient
surrogates are obtained by greedy center selection. We show that the
convergence analysis from the non-augmented setting carries over to the
product-kernel framework and derive corresponding prediction error bounds.
Numerical experiments for a pendulum with varying length and time-step size
and for a parameter-dependent discretized wave equation illustrate the
accuracy and structure-preserving behavior of the proposed approach.
\end{abstract}


\section{Introduction}

Learning flow maps directly from trajectory data has emerged as a powerful alternative to first identifying governing equations or vector fields and then integrating them numerically. Rather than approximating the right-hand side of a dynamical system, one learns the evolution operator itself, i.e., a map that propagates the state over a finite time interval. This viewpoint avoids numerical differentiation of data, naturally accommodates for coarse temporal sampling, and yields predictors that advance the system over large time horizons in a single step. Direct flow-map learning has been successfully used for unknown dynamical systems and for approximating evolution operators of differential equations; see, e.g., \cite{qin2019data,churchill2023flowmap,boffi2025selfdistill}. In this work, we learn flow maps that are tailored to the structure of Hamiltonian systems. Many conservative phenomena in areas such as classical mechanics, theoretical chemistry, and molecular dynamics admit such a formulation. In canonical coordinates \(x=(q,p)\in\R^{2n}\), with parameter \(\mu\in\mathcal P\), we study systems of the form
\begin{align}\label{eqn:Hamsys}
  \dot x(t)=J_{2n}\nabla_x \Ham(x(t);\mu),
  \qquad
  x(0)=x_0,
\end{align}
where
\[
J_{2n}:=
\begin{bmatrix}
0_n & I_n\\
-I_n & 0_n
\end{bmatrix}
\]
denotes the canonical Poisson matrix. For fixed \(\mu\in\mathcal P\) and \(\Delta T\ge 0\), we write
\[
\Phi_\mu^{\Delta T}(x_0):=x(\Delta T;x_0,\mu)
\]
for the associated flow map. 
A defining feature of Hamiltonian systems is that the associated flow maps are symplectic, i.e.,
\[
\bigl(D\Phi_\mu^{\Delta T}(x)\bigr)^\top J_{2n} D\Phi_\mu^{\Delta T}(x)=J_{2n}
\qquad\text{for all }x\in\R^{2n}.
\]
Symplectic maps possess fundamental geometric properties; see, e.g., \cite{daSilva2008}. In particular, they preserve phase-space volume, are locally invertible, and are closed under composition. These structural properties are closely tied to the favorable long-time behavior of Hamiltonian systems and should therefore also be respected by data-driven approximations of their flow maps.

In practice, exact flow maps are rarely available in closed form and must be approximated numerically. To preserve the geometric structure of the dynamics, one commonly employs symplectic integrators \cite{hairer2006}. A basic example is the symplectic Euler method,
\begin{align}\label{eq:symplEuler}
  x_{i+1}
  =
  x_i+\Delta t\,J_{2n}\nabla \Ham(q_i,p_{i+1}),
\end{align}
whose update map is symplectic for fixed \(\Delta t>0\) under suitable solvability assumptions; see \cite[Theorem~3.3]{hairer2006}. Moreover, backward-error analysis explains why symplectic time integrators often exhibit near-conservation of energy over very long time intervals \cite{hairer2006}. Even so, structure-preserving time-stepping can become computationally expensive when stability or accuracy demands require very small step sizes. This motivates the direct learning of large-step flow maps.

For Hamiltonian systems, the challenge is not only to learn an accurate predictor, but also to preserve the symplectic structure of the exact flow. A first line of research addresses this by learning an underlying Hamiltonian or vector field and then combining it with a structure-preserving integrator. Hamiltonian Neural Networks (HNNs) \cite{greydanus2019hamiltonian} learn a scalar Hamiltonian surrogate \(\Ham_{\mathrm{NN}}\) and recover the dynamics via
\[
\dot x = J_{2n}\nabla \Ham_{\mathrm{NN}}(x).
\]
Related approaches include Gaussian-process-based learning of Hamiltonian systems \cite{bertalan2019learning} and separable-Hamiltonian models combined with symplectic time-stepping \cite{Chen2020Symplectic}. A second line of work seeks to learn the Hamiltonian flow map itself as a symplectic map, without explicitly recovering a physical Hamiltonian. SympNets \cite{jin2020sympnets} approximate the time-stepping map by composing simple symplectic building blocks. H\'enonNets \cite{burby2020fast} pursue a related strategy based on concatenations of H\'enon-like maps. Generating Function Neural Networks (GFNNs) \cite{chen2021data} learn a generating function from which a symplectic map can be constructed directly. More recently, Generalized Hamiltonian Neural Networks (GHNNs) \cite{Horn2024} have provided a unifying perspective encompassing several earlier symplectic neural architectures. In more recent works such ideas have been extended to parameter- and time-dependent settings. Parametric Generalized Hamiltonian Neural Networks generalize the GHNN framework to parameter-dependent systems \cite{horn6427896parametric}, while time-adaptive SympNets address variable step sizes and time dependence \cite{janik2025time}. This makes parameter dependence and variable macro time-steps natural next steps in structure-preserving flow-map learning. Therefore, the present work extends our kernel-based approach \cite{herkert2026kernel} to parameter dependence and variable macro time-step sizes. Compared with neural network architectures, kernel methods offer several attractive features. They lead to reproducing kernel Hilbert space (RKHS)-based learning problems
with closed-form solutions, provide direct access to derivatives through
reproducing identities. Further, they come with a rigorous approximation theory covering
Hermite--Birkhoff (HB) interpolation and greedy strategies
\cite{wendland_2004}. Moreover, kernel methods have shown strong practical
performance in surrogate modeling tasks
\cite{carlberg2019recovering,doeppel2024goal}.
Hermite-type kernel interpolation arises in many applications, including PDE discretization \cite{la2008double} and image reconstruction \cite{de2018image}. 

In our recent work \cite{herkert2026kernel}, we constructed a kernel-based symplectic predictor by learning a differentiable scalar surrogate \(s:\R^{2n}\to\R\) and inserting its gradient into the implicit update
\begin{align}\label{eqn:pred_intro}
  x_{\Delta T,\mathrm{pred}}
  =
  x_0+\Delta T\,J_{2n}\nabla s(q_0,p_{\Delta T,\mathrm{pred}}),
\end{align}
where \(x_0=(q_0,p_0)\) and \(x_{\Delta T,\mathrm{pred}}=(q_{\Delta T,\mathrm{pred}},p_{\Delta T,\mathrm{pred}})\). Since this construction reflects the symplectic Euler method, whose update map
is symplectic under suitable solvability assumptions, the resulting update map 
is symplectic by construction. To train the surrogate, one starts from data pairs
\[
x_0^j=(q_0^j,p_0^j),\qquad
x_{\Delta T}^j=(q_{\Delta T}^j,p_{\Delta T}^j)=\Phi^{\Delta T}(x_0^j),
\qquad j=1,\dots,M,
\]
and introduces the mixed variables and targets
\[
\xi_j:=(q_0^j,p_{\Delta T}^j),
\qquad
y_j:=J_{2n}^\top\frac{x_{\Delta T}^j-x_0^j}{\Delta T}.
\]
With this choice, \eqref{eqn:pred_intro} results in the gradient relation
\begin{equation}\label{HBint}
\nabla s(\xi_j)=y_j,
\qquad j=1,\dots,M,  
\end{equation}
so that learning the surrogate reduces naturally to a gradient HB interpolation problem. In particular, the model is trained to reproduce the gradient information entering the symplectic update, rather than the Hamiltonian itself.

The present paper extends the kernel-based symplectic predictor of
\cite{herkert2026kernel} from a fixed, parameter-independent setting to a
parameter-augmented framework. The main contribution is the construction of a
single product-kernel surrogate on the augmented domain
\(\Omega_\xi\times\mathcal P\times\mathcal T\) instead of training separate models for each parameter and macro step setting.
The current approach represents a family of
large-step flow-map predictors depending on the physical parameter \(\mu\) and
the macro time-step size \(\Delta T\). For evaluating and training the predictor via \eqref{eqn:pred_intro}, differentiation is performed only with
respect to the mixed state variables, so that each fixed admissible pair
\((\mu,\Delta T)\) yields an implicit symplectic-Euler-type update. Consequently,
the learned predictor preserves symplecticity by construction for every
parameter and time-step instance, while allowing parameter- and step-size
dependent dynamics to be learned within one surrogate model.

The remainder of the paper is organized as follows. In \Cref{Sec:Background}, we provide an introduction to generalized kernel interpolation, greedy kernel methods and state the convergence results for the symplectic predictor. Numerical experiments are presented in \Cref{Sec:Numerics}, and we conclude with a summary and outlook in \Cref{Sec:Conclusion}.

\section{Background on Generalized Kernel Interpolation}\label{Sec:Background}
Let \(\mathcal X\) be a nonempty set. A kernel on \(\mathcal X\) is a symmetric function $k:\mathcal X\times\mathcal X\to\R.$
Given a finite set of pairwise distinct points
$Z_M:=\{z_1,\dots,z_M\}\subset\mathcal X,$
the matrix
\[
G_{Z_M}:=\bigl(k(z_i,z_j)\bigr)_{i,j=1}^M\in\R^{M\times M}
\]
is called the Gramian matrix of \(k\) with respect to \(Z_M\). We call \(k\) positive definite (p.d.) if, for every \(M\in\N\) and every finite set
\(Z_M\subset\mathcal X\), the matrix \(G_{Z_M}\) is positive semidefinite. We call \(k\)
strictly positive definite (s.p.d.) if, for every \(M\in\N\) and every finite set
\(Z_M\subset\mathcal X\) consisting of pairwise distinct points, the matrix \(G_{Z_M}\) is positive definite. An RKHS \(H_k(\mathcal X)\) over \(\mathcal X\) is a Hilbert space of functions
\(f:\mathcal X\to\R\) for which all point-evaluation functionals are continuous. For every RKHS there exists a unique function $k$, the so called reproducing kernel, such that \(k(\cdot,z)\in H_k(\mathcal X)\) and
\[
f(z)=\langle f,k(\cdot,z)\rangle_{H_k(\mathcal X)}
\qquad\text{for all }f\in H_k(\mathcal X),\ z\in\mathcal X.
\]
Conversely, every p.d.\ kernel induces a unique RKHS, so that p.d.\ kernels and RKHSs are in one-to-one correspondence. 

In the following, we will recast the interpolation problem \eqref{HBint} into a generalized interpolation problem in an RKHS \(H_k(\mathcal X)\) which is given by 
\begin{align}\label{eq:min_norm_problem}
  s_M
  =
  \argmin_{s \in H_k(\mathcal X)}
  \left\{
   \|s\|_{H_k(\mathcal X)}
   \ \big|\
   \lambda_j(s)=y_j \text{ for } j=1,\dots,M
  \right\}.
\end{align}
Here, \(\lambda_1,\dots,\lambda_M \in H_k(\mathcal X)'\) are linearly independent continuous linear functionals, and \(y_1,\dots,y_M\in\R\) are prescribed target values. By the Riesz representation theorem, for each \(\lambda_i\) there exists a unique \(v_i\in H_k(\mathcal X)\) such that
\[
\lambda_i(f)=\langle f,v_i\rangle_{H_k(\mathcal X)}
\qquad\text{for all }f\in H_k(\mathcal X).
\]
If \(k\) is the reproducing kernel of \(H_k(\mathcal X)\), then \(v_i\) admits the explicit representation
\[
v_i(z)
=
\bigl(\lambda_i^{(2)}k\bigr)(z)
=
\lambda_i\bigl(k(z,\cdot)\bigr),
\]
that is, \(\lambda_i\) acts on the second argument of \(k\). For the vector of functionals
$\Lambda_M=[\lambda_1,\dots,\lambda_M],$
we define the generalized Gramian matrix \(G_\Lambda\in\R^{M\times M}\) by
\[
(G_\Lambda)_{ij}
:=
\langle v_j,v_i\rangle_{H_k(\mathcal X)}
=
\lambda_i^{(1)}\lambda_j^{(2)}k,
\qquad i,j=1,\dots,M.
\]
By \cite[Theorem~16.1]{wendland_2004}, if \(\lambda_1,\dots,\lambda_M\) are linearly independent on \(H_k(\mathcal X)\), then for every \(y\in\R^M\) there exists a unique minimum-norm interpolant in the sense of \eqref{eq:min_norm_problem}, and it is given by
\[
s(\cdot)
=
\sum_{j=1}^M c_j\,v_j(\cdot)
=
\sum_{j=1}^M c_j\,\lambda_j^{(2)}k(\cdot,\cdot),
\qquad\text{where}\qquad
G_\Lambda c = y \qquad\text{with}\qquad
c:=(c_j)_{j=1}^M.
\]

\medskip

A special case of this generalized interpolation problem is the HB interpolation underlying our training procedure. In the present parameter-augmented setting, let
\[
Z_M=\{z_j\}_{j=1}^M\subset\mathcal X,
\qquad
z_j=(\xi_j,\mu_j,{\Delta T}_j),
\qquad
z_j\in\mathcal X,
\]
and let
\[
\alpha_j\in\{1,\dots,2n\},
\qquad j=1,\dots,M,
\]
be state-coordinate indices. We then define the sampling functionals
\begin{equation}\label{eqn:sampling}
\lambda_j(f)
:=
\lambda_{j,\alpha_j}(f)
:=
\partial_{\xi_{\alpha_j}}f(z_j),
\qquad j=1,\dots,M.
\end{equation}
Hence, the interpolation conditions involve derivatives only with respect to the state variables.

In the present work, the surrogate depends on mixed state variables as well as on additional parameter and time-step variables. More precisely, we consider an augmented domain of the form
\begin{equation}\label{def:omega}
\Omega := \Omega_{\xi} \times \mathcal P \times \mathcal T,  
\end{equation}
where \(\Omega_{\xi} \subset \R^{2n}\) denotes the domain of mixed state variables
\(\xi=(q_0,p_{\Delta T})\), \(\mathcal P \subset \R^{d_\mu}\) is a parameter domain, and
\(\mathcal T \subset (0,\infty)\) is the macro time-step domain. To incorporate physical parameters and macro time-step sizes into the interpolation problem, we employ a product kernel on \(\Omega\) of the form
\begin{equation}\label{eq:product_kernel_background}
k\bigl((\xi,\mu,{\Delta T}),(\xi',\mu',{\Delta T}')\bigr)
=
k_{\xi}(\xi,\xi')\,k_\mu(\mu,\mu')\,k_{\Delta T}({\Delta T},{\Delta T}'),
\end{equation}
where \(k_{\xi}\) is a kernel on the mixed-state domain \(\Omega_{\xi}\), and \(k_\mu\), \(k_{\Delta T}\) are kernels on \(\mathcal P\) and \(\mathcal T\), respectively. If the kernel factors \(k_\xi\), \(k_\mu\), and \(k_{\Delta T}\) are positive
definite on their respective domains, then the product kernel $k$ is positive definite on the augmented domain
\(\Omega_\xi\times\mathcal P\times\mathcal T\). In the considered setting, differentiation is taken only with respect to the state block
\(\xi\in\Omega_{\xi}\subset\R^{2n}\), while \((\mu,{\Delta T})\in\mathcal P\times\mathcal T\) enter as additional arguments. If \(k\in C^2(\Omega\times\Omega)\), then the first-order partial derivative point-evaluation functionals with respect to the state variables are continuous and satisfy the derivative reproducing property
\begin{equation}\label{eqn:rep_deriv}
\partial_{\xi_\ell} f(\xi,\mu,{\Delta T})
=
\big\langle
f,\,
\partial_{\xi_\ell}^{(2)}k\bigl(\cdot,(\xi,\mu,{\Delta T})\bigr)
\big\rangle_{H_k(\Omega)}
\quad
\text{for all }f\in H_k(\Omega),\ (\xi,\mu,{\Delta T})\in\Omega,\ \ell=1,\dots,2n,
\end{equation}
where \(\partial_{\xi_\ell}^{(2)}\) denotes differentiation of \(k\) with respect to the \(\ell\)-th state coordinate of its second argument.
If \(k\in C^2(\Omega\times\Omega)\), then these functionals are continuous on
\(H_k(\Omega)\) by \eqref{eqn:rep_deriv} and the Cauchy--Schwarz inequality. In the following, we assume that the
HB sampling functionals \(\lambda_1,\dots,\lambda_M\) are linearly independent
on \(H_k(\Omega)\). Under this assumption, the generalized
interpolation problem is well posed for the product-kernel space.

For a product kernel of the form \eqref{eq:product_kernel_background}, the derivative representers factorize as
\[
\partial_{\xi_\ell}^{(2)}
k\bigl((\xi,\mu,{\Delta T}),(\xi',\mu',{\Delta T}')\bigr)
=
k_\mu(\mu,\mu')\,k_{\Delta T}({\Delta T},{\Delta T}')\,
\partial_{\xi_\ell}^{(2)}k_{\xi}(\xi,\xi').
\]
Thus, the parametric dependence appears only through multiplicative kernel factors, whereas the derivative structure is entirely determined by the state kernel \(k_\xi\). Moreover, the Riesz representers of the HB functionals are then given by
\[
v_j(\xi,\mu,{\Delta T})
=
\partial_{\xi_{\alpha_j}}^{(2)}
k\bigl((\xi,\mu,{\Delta T}),(\xi_j,\mu_j,{\Delta T}_j)\bigr)
=
k_\mu(\mu,\mu_j)\,k_{\Delta T}({\Delta T},{\Delta T}_j)\,
\partial_{\xi_{\alpha_j}}^{(2)}k_{\xi}(\xi,\xi_j),
\]
and the corresponding generalized Gramian entries take the form
\[
(G_\Lambda)_{ij}
=
\partial_{\xi_{\alpha_i}}^{(1)}\partial_{\xi_{\alpha_j}}^{(2)}k(z_i,z_j)
=
k_\mu(\mu_i,\mu_j)\,k_{\Delta T}({\Delta T}_i,{\Delta T}_j)\,
\partial_{\xi_{\alpha_i}}^{(1)}\partial_{\xi_{\alpha_j}}^{(2)}k_{\xi}(\xi_i,\xi_j).
\]
Hence, the interpolant can be written as
\begin{equation}\label{eqn:surr}
s(z)
=
\sum_{j=1}^M c_j\,
\partial_{\xi_{\alpha_j}}^{(2)}k(z,z_j),
\qquad z=(\xi,\mu,{\Delta T})\in\Omega,
\end{equation}
where the coefficients \(c_j\), \(j=1,\dots,M\), are determined by the linear system
$G_\Lambda c = y,$ with $c:=(c_j)_{j=1}^M.$

\subsection{Sparse greedy approximation}

For large \(M\), solving the dense system \(G_\Lambda c=y\) and evaluating \eqref{eqn:surr} can become computationally expensive and numerically ill-conditioned. A common strategy is therefore to replace the full set of functionals by a much smaller subset
\[
\Lambda_{m_{\max}}\subset\Lambda_M,
\qquad
m_{\max}\ll M.
\]
Greedy algorithms, such as VKOGA~\cite{Santin2021VKOGA,wirtz2013vectorial}, construct such subsets iteratively according to a selection rule based on the current interpolant. A common choice is the \(f\)-greedy rule. Starting from \(\Lambda_0=\emptyset\) and \(s_{f,\Lambda_0}=0\), one selects at iteration \(m\ge 0\)
\begin{equation}\label{eq:f_greedy}
 \lambda_{m+1}
 \in \operatorname*{argmax}_{\lambda_j \in \Lambda_M}
 \left|y_j-\lambda_j(s_{f,\Lambda_m})\right|,
 \qquad
 \Lambda_{m+1}:=\Lambda_m\cup\{\lambda_{m+1}\},
\end{equation}
where the residuals at previously selected functionals vanish, i.e.,
\[
y_j-\lambda_j(s_{f,\Lambda_m})=0,
\qquad j=1,\dots,m.
\]
In the present HB setting, this amounts to iteratively selecting state-derivative evaluations at augmented sample points
\[
(z_i,\ell_i)\in \Omega\times\mathcal J,
\qquad
\mathcal J:=\{1,\dots,2n\},
\]
where \(\ell_i\) denotes a state-coordinate index. For a finite selection
\[
\{(z_i,\ell_i)\}_{i=1}^m\subset\Omega\times\mathcal J,
\]
and with $V_m
:=
\Span\left\{
\partial_{\xi_{\ell_i}}^{(2)}k(\cdot,z_i)
:
i=1,\dots,m
\right\}$
we define
\begin{equation}\label{def:projector}
\Pi_m:H_k(\Omega)\to V_m,
\end{equation}
to be the orthogonal projector onto \(V_m\).
Given a target function \(u\in H_k(\Omega)\), we set
\[
s_m:=\Pi_m u,
\qquad
e_m:=u-s_m.
\]
In an idealized continuous formulation, the \(f\)-greedy selection rule can be
written as
\[
(z_{m+1},\ell_{m+1})
\in
\operatorname*{argmax}_{z\in\Omega,\ \ell\in\mathcal J}
\left|\partial_{\xi_\ell}e_m(z)\right|.
\]
In the computational setting, this maximization is restricted to the finite
candidate set of available sampling functionals. After selecting
\((z_{m+1},\ell_{m+1})\), we set \(s_{m+1}:=\Pi_{m+1}u\).

Greedy schemes yield sparse surrogates and come with rigorous, in some settings even optimal, convergence guarantees~\cite{wenzel2023analysis,santin2024optimality}. Recent refinements of this theory sharpen the known convergence rates for target-data-dependent greedy generalized interpolation with Sobolev kernels by removing an additional logarithmic factor in the error bounds \cite{Santin2026Ref}. For HB interpolation, greedy convergence has been for example studied in
\cite{ehring2024hermite} where convergence of a target-dependent greedy Hermite kernel scheme based on an \(f\)-type selection rule for value-function surrogates is proven in optimal control. Complementarily, \cite{albrecht2025convergence} develops a convergence theory for generalized kernel-based interpolation with totally bounded sets of sampling functionals, which in particular covers HB interpolation. For the non-augmented symplectic predictor, convergence rates for gradient-HB interpolation were established in \cite{herkert2026kernel}. The present paper extends this analysis to the parameter-augmented setting by working on the product domain \(\Omega\) and differentiating only with respect to the state variables.

To have a well-defined 
interpolation problem in \eqref{HBint} (in the sense that existence of the target function is ensured), one needs the notion of a type~II generating function. 
Therefore, we extend the notion of a type~II generating function to the augmented domain \(\Omega\). For background on generating functions, see \cite{cline2017variational}.

\begin{definition}[Type~II generating function]\label{def:typeII}
Let \(\mathcal P\subset\R^{d_\mu}\) and \(\mathcal T\subset(0,\infty)\) be parameter and time-step domains, and let
\[
\Phi_\mu^{\Delta T} : D_{\mu,{\Delta T}}\to\R^{2n}
\]
be a family of Hamiltonian flow maps. Define
$I_1:=
\begin{bmatrix}
I_n & 0\\
0 & 0
\end{bmatrix}, 
I_2:=
\begin{bmatrix}
0 & 0\\
0 & I_n
\end{bmatrix}$
and let $I_1x_0+I_2\Phi_\mu^{\Delta T}(x_0) \in \Omega.$
A function
\[
S\in H_k(\Omega)\cap C^1(\Omega),
\qquad
\Omega\subset \R^{2n}\times\mathcal P\times\mathcal T,
\]
is called a type~II generating function with parameter and time-step variables if, for every admissible
\((x_0,\mu,{\Delta T})\) with \(x_0\in D_{\mu,{\Delta T}}\),
\begin{equation}\label{eqn:mixed-training_augm}
J_{2n}^{\top}
\frac{\Phi_\mu^{\Delta T}(x_0)-x_0}{{\Delta T}}
=
\nabla_\xi S\!\left(I_1x_0+I_2\Phi_\mu^{\Delta T}(x_0),\,\mu,{\Delta T}\right),
\end{equation}
where the gradient is taken only with respect to the mixed state variable
\(\xi=(q_0,p_{\Delta T})\).
\end{definition} 
In \cite{herkert2026kernel} we provide conditions under which such a type~II generating function exists in the non-augmented setting. 
We next state the parameter-augmented analogues of the non-augmented convergence results. The first proposition gives a convergence estimate for the state-gradient approximation produced by the gradient-HB \(f\)-greedy procedure on the augmented domain. The second proposition shows how this gradient error propagates to the prediction error of the parameter-augmented symplectic predictor.

\begin{proposition}[Convergence rate for parameter-augmented gradient-HB interpolation]
\label{thm:HB-augm}
Let \(\Omega\) and \(k\) be defined as in \eqref{def:omega} and
\eqref{eq:product_kernel_background}. Assume that \(k\) is twice continuously
differentiable with respect to the state variables in both arguments and that
\(S\in H_k(\Omega)\). Let \(s_m\in H_k(\Omega)\) be obtained by the
gradient-HB \(f\)-greedy procedure on the augmented domain, where the sampling
functionals are defined as in \eqref{eqn:sampling}. Let
\(\Pi_m:H_k(\Omega)\to V_m\) denote the orthogonal projector defined in
\eqref{def:projector}.

For a candidate state-derivative functional
\[
\lambda_{z,\ell}(f):=\partial_{\xi_\ell}f(z),
\qquad
z\in\Omega,\quad \ell\in\{1,\dots,2n\},
\]
define the associated generalized power function by
\[
P_{\Lambda_m}(z,\ell)
:=
\left\|
\lambda_{z,\ell}^{(2)}k
-
\Pi_m\bigl(\lambda_{z,\ell}^{(2)}k\bigr)
\right\|_{H_k(\Omega)} .
\]
For \(m\ge 1\), choose an index \(i_m\in\{m+1,\dots,2m\}\) such that
\[
\|\nabla_\xi e_{i_m}\|_{L^\infty(\Omega)}
=
\min_{m+1\le i\le 2m}
\|\nabla_\xi e_i\|_{L^\infty(\Omega)},
\qquad
e_i:=S-s_i .
\]
Then
\[
\|\nabla_\xi e_{i_m}\|_{L^\infty(\Omega)}
\le
\sqrt{2n}\,m^{-1/2}\,\|e_{m+1}\|_{H_k(\Omega)}
\left[
\prod_{i=m+1}^{2m}
P_{\Lambda_i}(z_{i+1},\ell_{i+1})
\right]^{1/m}.
\]
\end{proposition}

\begin{proof}
The argument is the same projection-based \(f\)-greedy argument as in the
non-augmented gradient-HB setting, applied on the augmented domain
\(\Omega\). The sampling functionals are state-derivative evaluations, and
their Riesz representers are computed in the product-kernel RKHS
\(H_k(\Omega)\). Thus, the greedy residuals, orthogonal projections, and
generalized power functions have the same structure as in the
non-augmented case. Since the proof only uses this RKHS projection structure
and the corresponding power-function estimates, it carries over directly to
the present parameter-augmented setting.
\end{proof}

\begin{remark}
In the product-kernel setting, the parameter and time-step variables influence the convergence behavior through the generalized power function on the augmented domain.
\end{remark}

\begin{proposition}[Convergence rate for the parameter-augmented prediction error]\label{thm:pred-augm}
Let
\[
S\in C^2(\Omega)\cap H_k(\Omega)
\]
be a type~II generating function with parameter and time-step variables in the sense of Definition~\ref{def:typeII}, and let \(s_{i_m}\in H_k(\Omega)\) be the gradient-HB \(f\)-greedy interpolant from Proposition~\ref{thm:HB-augm}. Fix a compact set
\[
K\subset \R^{2n}\times\mathcal P\times\mathcal T
\]
and define
\[
{\Delta T}_{\max}:=\sup\{{\Delta T}:(x,\mu,{\Delta T})\in K\}.
\]
Assume the following:

\begin{enumerate}
\item[(i)] \textbf{Uniform solvability.}
For every \((x_0,\mu,{\Delta T})\in K\) and every \(m\) sufficiently large, the implicit equations
\[
y=x_0+{\Delta T} J_{2n}\nabla_\xi S(I_1x_0+I_2y,\mu,{\Delta T}),
\qquad
y=x_0+{\Delta T} J_{2n}\nabla_\xi s_{i_m}(I_1x_0+I_2y,\mu,{\Delta T})
\]
admit unique solutions \(y^\ast=\Phi_\mu^{\Delta T}(x_0)\) and \(y_m=x_{{\Delta T},i_m}(x_0,\mu)\), respectively.

\item[(ii)] \textbf{Uniform Lipschitz continuity of \(\nabla_\xi S\) with respect to \(y\).}
There exists a constant \(L_S>0\) such that
\[
\sup_{\substack{(x_0,\mu,\Delta T)\in K\\
y:\ (I_1x_0+I_2y,\mu,\Delta T)\in\Omega}}
\bigl\|
\nabla_\xi^2 S(I_1x_0+I_2y,\mu,\Delta T)\,I_2
\bigr\|_2
=: L_S <\infty .
\]
\end{enumerate}
Assume further that
\begin{equation}\label{eq:tau_condition_augm}
{\Delta T}_{\max}<\frac{1}{L_S}.
\end{equation}
Then, with $C_{\mathrm{pred}}
:=
\frac{1}{1-{\Delta T}_{\max}L_S},$
one has
\[
\sup_{(x_0,\mu,{\Delta T})\in K}
\|x_{{\Delta T},i_m}(x_0,\mu)-\Phi_\mu^{\Delta T}(x_0)\|_2
\le
C_{\mathrm{pred}}\,{\Delta T}_{\max}\,
\|\nabla_\xi e_{i_m}\|_{L^\infty(\Omega)}.
\]
In particular, by Proposition~\ref{thm:HB-augm},
\[
\sup_{(x_0,\mu,{\Delta T})\in K}
\|x_{{\Delta T},i_m}(x_0,\mu)-\Phi_\mu^{\Delta T}(x_0)\|_2
\le
C\,{\Delta T}_{\max}\,\sqrt{2n}\,m^{-1/2}\,\|e_{m+1}\|_{H_k(\Omega)}
\left[
\prod_{i=m+1}^{2m}
P_{\Lambda_i}(z_{i+1},\ell_{i+1})
\right]^{1/m},
\]
for some constant \(C>0\) independent of \(m\).
\end{proposition}
\begin{proof}
The proof follows the argument from the non-augmented case, applied
uniformly on the compact set \(K\). For fixed
\((x_0,\mu,\Delta T)\in K\), let
\[
y^\ast:=\Phi_\mu^{\Delta T}(x_0),
\qquad
y_m:=x_{\Delta T,i_m}(x_0,\mu).
\]
Subtracting the two implicit equations, we decompose the resulting gradient
difference into the surrogate approximation error and the variation of the
exact gradient with respect to the implicit variable. More precisely,
\[
\begin{aligned}
y_m-y^\ast
=
\Delta T
\Big[
&
\nabla_\xi S_{i_m}(I_1x_0+I_2y_m,\mu,\Delta T)
-
\nabla_\xi S(I_1x_0+I_2y_m,\mu,\Delta T)
\\
&+
\nabla_\xi S(I_1x_0+I_2y_m,\mu,\Delta T)
-
\nabla_\xi S(I_1x_0+I_2y^\ast,\mu,\Delta T)
\Big].
\end{aligned}
\]
Taking norms and using the uniform gradient error bound together with the
Lipschitz continuity of \(\nabla_\xi S\) with respect to \(y\), we obtain
\[
\|y_m-y^\ast\|_2
\le
\Delta T\|\nabla_\xi e_{i_m}\|_{L^\infty(\Omega)}
+
\Delta T L_S\|y_m-y^\ast\|_2 .
\]
Since \(\Delta T\le \Delta T_{\max}<1/L_S\), this yields
\[
\|y_m-y^\ast\|_2
\le
\frac{\Delta T}{1-\Delta T_{\max}L_S}
\|\nabla_\xi e_{i_m}\|_{L^\infty(\Omega)}
\le
C_{\mathrm{pred}}\Delta T_{\max}
\|\nabla_\xi e_{i_m}\|_{L^\infty(\Omega)}.
\]
Taking the supremum over \(K\) gives the first estimate, and the second follows
by inserting the bound from Proposition~\ref{thm:HB-augm}.
\end{proof}

\section{Numerical Experiments}\label{Sec:Numerics}
We evaluate the proposed parameter-augmented symplectic kernel predictor on benchmark Hamiltonian systems and compare its predictions with high-fidelity reference solutions. Throughout, reference trajectories are computed by the implicit midpoint rule with a sufficiently small micro time-step \(\Delta t = 10^{-3}\). The numerical experiments cover two settings: parameter-dependent problems with fixed macro time-step and fully augmented problems in which the macro time-step is treated as an additional variable. In particular, some experiments (\Cref{sec:wave}) are performed for fixed \(\Delta T\), whereas others (\Cref{sec:pendulum_param_time}) include \(\Delta T\) as an additional input variable in the surrogate model. Accordingly, the predictor is trained either on the augmented state--parameter domain \(\Omega_\xi\times\mathcal P\) or on the full augmented domain \(\Omega_\xi\times\mathcal P\times\mathcal T\). We report the following error measures.

\paragraph{(i) Maximum residual error versus number of centers.}
Let \(Z_{\mathrm{train}}\) and \(Z_{\mathrm{val}}\) denote the training and validation sets in the relevant augmented domain. For
\(Z\in\{Z_{\mathrm{train}},Z_{\mathrm{val}}\}\), we define
\[
E_Z(m)
:=
\max_{\ell\in\mathcal J}\ \max_{z\in Z}
\left|\partial_{\xi_\ell} e_m(z)\right|,
\]
that is, the worst-case mismatch in the state gradient over the corresponding data set after \(m\) greedy steps.

\paragraph{(ii) Relative trajectory error over time.}
For a test instance \((x_0,\mu,{\Delta T})\), we measure the deviation of the predicted trajectory from the reference solution on the macro time grid \(t_k=k\Delta T\) by
\[
e_{\mathrm{rel}}(t_k;x_0,\mu,\Delta T)
:=
\frac{\|x_{\mathrm{pred}}(t_k;x_0,\mu,\Delta T)-x_{\mathrm{ref}}(t_k;x_0,\mu)\|_2}
   {\|x_{\mathrm{ref}}(t_k;x_0,\mu)\|_2},
\]
where \(x_{\mathrm{ref}}(t_k;x_0,\mu)\) denotes the reference states, and \(x_{\mathrm{pred}}(t_k;x_0,\mu,\Delta T)\) are the states produced by the symplectic kernel predictor. In experiments with fixed \(\Delta T\), we suppress the dependence on \(\Delta T\) in the notation.

\medskip

Both the kernel family and the associated shape parameters are selected by validation. In the fully augmented setting, we consider product kernels of the form
\[
k\bigl((\xi,\mu,\Delta T),(\xi',\mu',\Delta T')\bigr)
=
k_\xi(\xi,\xi';\varepsilon_\xi)\,k_\mu(\mu,\mu';\varepsilon_\mu)\,k_{\Delta T}(\Delta T,\Delta T';\varepsilon_{\Delta T}).
\] 
For experiments with fixed \(\Delta T\), this reduces to a product kernel where the corresponding factor is omitted. Among a prescribed grid of kernel families and shape parameters, we choose the tuple minimizing the validation objective \(E_{Z_{\mathrm{val}}}(m^\star)\) for a fixed \(m^\star\). For each factor of the product kernel \eqref{eq:product_kernel_background}, we consider radial kernels of the form
\[
\hat k(\xi,\xi';\varepsilon)=\kappa(\varepsilon r),
\qquad
r=\|\xi-\xi'\|_2,
\]
including the inverse multiquadric (IMQ), Matérn \(3/2\), and Matérn \(5/2\)
kernels:
\[
\kappa_{\mathrm{IMQ}}(r)
=
\frac{1}{\sqrt{1+r^2}},
\qquad
\kappa_{\mathrm{M32}}(r)
=
(1+r)\exp(-r),
\qquad
\kappa_{\mathrm{M52}}(r)
=
\left(1+r+\tfrac13 r^2\right)\exp(-r).
\]

\subsection{Parameter-dependent discretized wave equation}\label{sec:wave}

We consider the one-dimensional wave equation on a finite interval. For \(t\in I=(0,T)\) and spatial variable \(\zeta\in D:=(0,L)\), we seek
\[
u:\bar I\times\bar D\to\R
\]
such that
\begin{align*}
u_{tt}(t,\zeta) &= c^2 u_{\zeta\zeta}(t,\zeta) && \text{in } I\times D,\\
u(t,\zeta) &= 0 && \text{on } I\times\partial D,\\
u(0,\zeta) &= u_0(\zeta),\quad
u_t(0,\zeta)=v_0(\zeta) && \text{in } D,
\end{align*}
where the wave speed \(c\) is treated as the physical parameter. In the experiments, we set \(l=1\) and \(T=6.0\), and use
\[
\mu:=c^2
\]
as the parameter in the surrogate. The training wave speeds are
\[
c_{\mathrm{train}}=\{0.20,0.22,0.24,0.26,0.28,0.30\},
\qquad
\mu_{\mathrm{train}}=\{c^2:\ c\in c_{\mathrm{train}}\}.
\]

We discretize \(D\) by a uniform grid with \(N=1000\) interior points and impose homogeneous Dirichlet boundary conditions at \(\zeta\in\{0,L\}\). Let \(D_{\zeta\zeta}\in\R^{N\times N}\) denote the standard symmetric positive definite central-difference matrix for \(-\partial_{\zeta\zeta}\). With \(u(t)\in\R^N\) denoting the vector of nodal values, the semi-discrete system becomes
\[
\ddot u(t)=- \mu D_{\zeta\zeta}u(t).
\]
Introducing the phase-space state
\[
x(t)=
\begin{bmatrix}
q(t)\\ p(t)
\end{bmatrix}
=
\begin{bmatrix}
u(t)\\ \dot u(t)
\end{bmatrix}
\in\R^{2N},
\]
we obtain the quadratic Hamiltonian
\[
\Ham(x;\mu)
=
\tfrac12 p^\top p+\tfrac12 \mu q^\top D_{\zeta\zeta}q,
\]
and hence the canonical Hamiltonian system
\[
\dot x(t)
=
J_{2N}H(\mu)x(t),
\qquad
J_{2N}=
\begin{bmatrix}
0&I_N\\
-I_N&0
\end{bmatrix},
\qquad
H(\mu)=
\begin{bmatrix}
\mu D_{\zeta\zeta} & 0\\[2pt]
0 & I_N
\end{bmatrix}.
\]
To reduce the \(2N\)-dimensional system while preserving symplectic structure, we employ symplectic model order reduction \cite{benner2020model,afkham2017structure}. We project onto a symplectic subspace spanned by
\[
V\in\R^{2N\times 2n},
\qquad
n\ll N,
\qquad
V^\top J_{2N}V=J_{2n},
\]
with symplectic inverse
\[
V^+:=J_{2n}^\top V^\top J_{2N}.
\]
The reduced coordinates \(x_{\mathrm{red}}(t)=V^+x(t)\) then satisfy
\[
\dot x_{\mathrm{red}}(t)
=
J_{2n}\bigl(V^\top H(\mu)V\bigr)x_{\mathrm{red}}(t),
\qquad
x_{\mathrm{red}}(0)=V^+x(0),
\]
with reduced Hamiltonian
\[
\Ham_{\mathrm{red}}(x_{\mathrm{red}};\mu)
=
\tfrac12 x_{\mathrm{red}}^\top\bigl(V^\top H(\mu)V\bigr)x_{\mathrm{red}}.
\]
The basis \(V\) is constructed by the complex SVD (cSVD) from snapshot data of a full-order reference model at $
c_{\mathrm{ref}}=0.3$, using low-frequency, energy-bounded initial states as in \cite{herkert2026kernel}. We retain the reduced dimension $
n=2,$ so that the reduced phase space has dimension \(2n=4\).

After constructing the basis, we generate training data directly in reduced coordinates. For each fixed macro time-step $
\Delta T\in\{0.1,0.05,0.025\}$
and each \(\mu\in\mu_{\mathrm{train}}\), we sample $N_s=20000$ reduced initial states from an energy-bounded subset of the reduced phase space. For each sample \((x_0,\mu)\), we compute the flow $\Phi_\mu^{\Delta T}(x_0)$
and form the data set
\[
\mathcal M_{\mathcal P}^{\Delta T}
:=
\left\{
\left(
\bigl(I_1 x_0 + I_2 \Phi_\mu^{\Delta T}(x_0),\,\mu\bigr),\,
J_{2n}^{\top}\frac{\Phi_\mu^{\Delta T}(x_0)-x_0}{\Delta T}
\right)
:
\mu\in\mu_{\mathrm{train}},\ x_0\in X_{\mathrm{train}}(\mu)
\right\}.
\]
Accordingly, the surrogate is trained on \(\Omega_\xi\times\mathcal P\), where $\xi=(q_0,p_{\Delta T})\in\R^{2n}.$

For testing, we draw for each test wave speed \(c_{\mathrm{test}}\) five initial conditions independently from the same energy-bounded region of reduced phase space. Each test trajectory is evolved up to \(T=6.0\) using
(i) the proposed symplectic kernel predictor,
(ii) the implicit midpoint rule with macro time-step \(\Delta T\), and
(iii) the micro-step reference solution. The relative errors are averaged over the five test initial conditions for each pair \((\Delta T,c_{\mathrm{test}})\).

As a first result, \Cref{fig:fgreedy-convergence-wave} summarizes both the interpolation and prediction performance. The left panel shows the decay of the \(f\)-greedy training and validation residuals as functions of the number of selected centers for the three macro time-step sizes, while the right panel compares the corresponding average relative trajectory errors of the learned symplectic predictor with those of the implicit-midpoint baseline for randomly chosen test wave speeds.
\begin{figure}[h]
 \centering
 \begin{subfigure}[t]{0.64\textwidth}
  \centering
  \includegraphics[width=\textwidth]{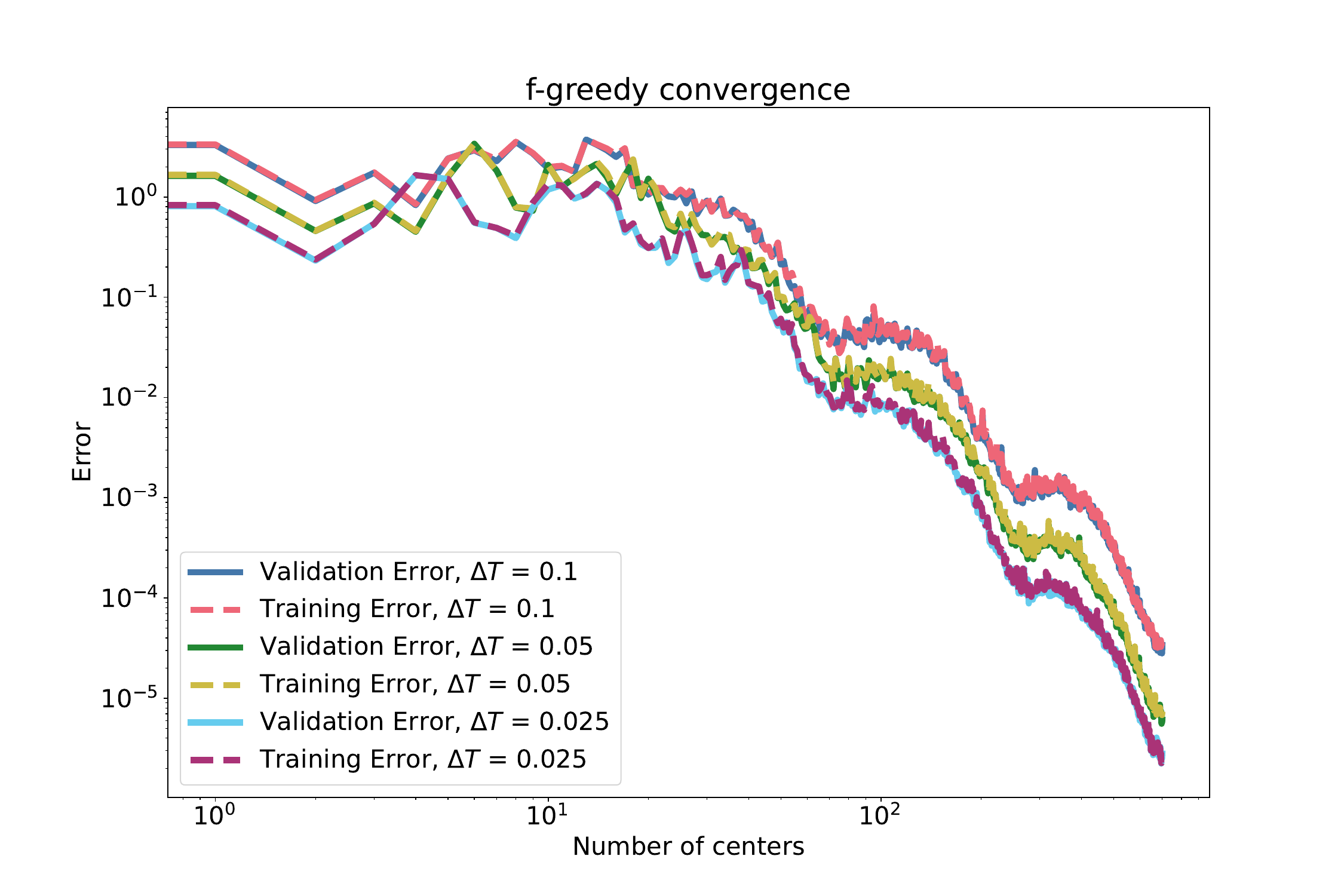}
  \caption{$f$-greedy interpolation error versus the number of selected centers (wave model), training (solid) and validation (dashed).}
  \label{fig:fgreedy-convergence-wave}
 \end{subfigure}\hfill
 \begin{subfigure}[t]{0.35\textwidth}
  \centering
  \includegraphics[width=\textwidth]{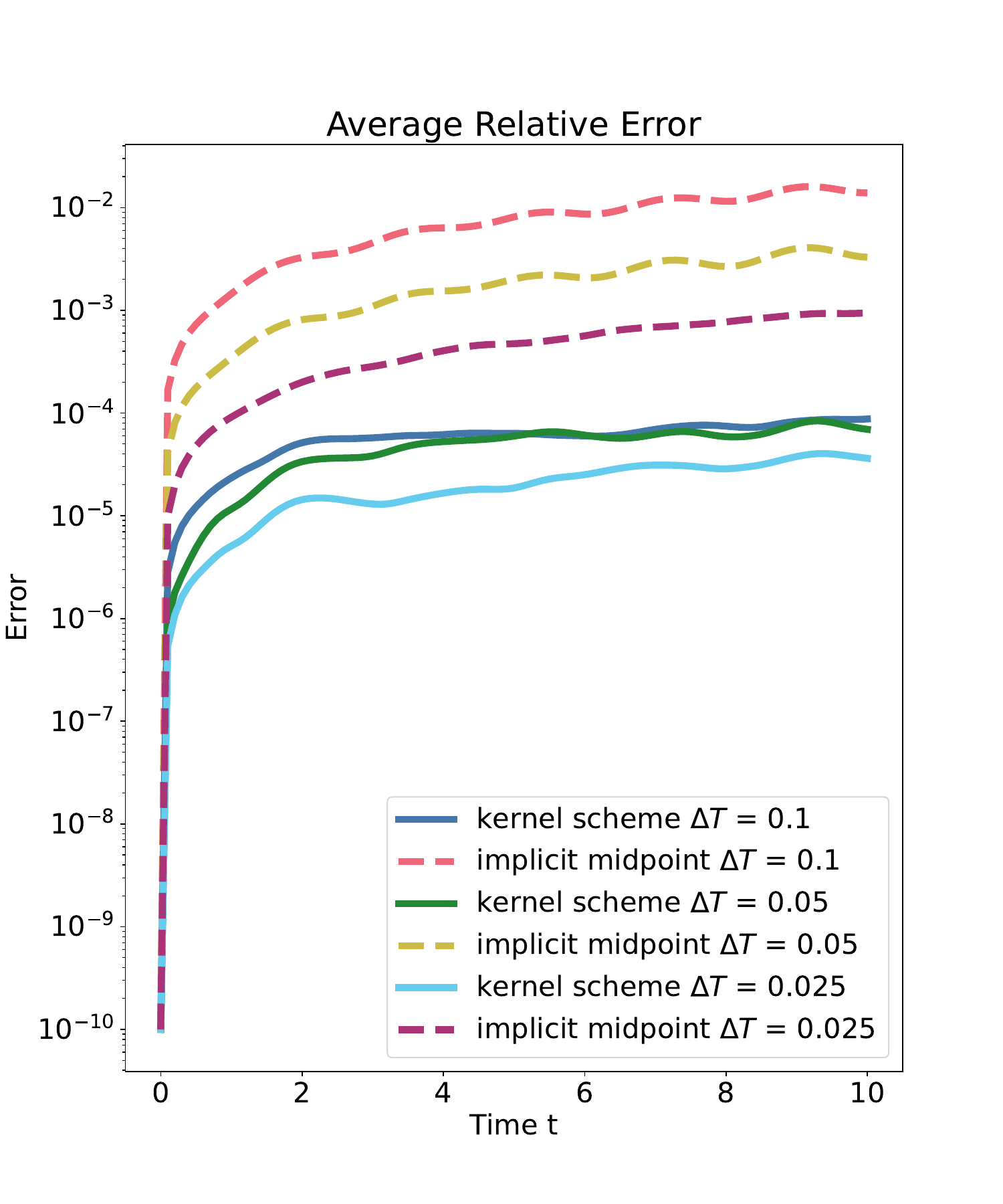}
  \caption{Average relative error over time (wave model), kernel predictor (solid) vs.\ implicit midpoint (dashed).}
  \label{fig:rel-reduction-error-wave}
 \end{subfigure}
 \caption{Reduced 1D wave model: (a) \(f\)-greedy residual over
the number of selected centers; (b) average relative trajectory error over
time.}
\end{figure} The training and validation residuals in \Cref{fig:fgreedy-convergence-wave} exhibit the same overall qualitative behavior for all three macro time-steps: after a short pre-asymptotic phase for small numbers of centers, the errors decay steadily over several orders of magnitude, with the validation curves remaining close to the corresponding training curves throughout. This indicates stable approximation behavior across the parameter range used for training and suggests a convergence behavior that is at least algebraic decay.
The averaged trajectory errors are shown in \Cref{fig:rel-reduction-error-wave}. For each macro time-step and each randomly chosen test wave speed, the kernel predictor consistently yields smaller errors than the implicit midpoint rule with macro time step $\Delta T$. Overall, the figure shows that the product-kernel surrogate captures the
reduced wave dynamics accurately over the tested parameter interval and
provides a clear improvement over direct macro-step integration by the
implicit midpoint rule.

\subsection{Pendulum with varying length and time-step size}\label{sec:pendulum_param_time}

As a test of the full augmented setting, we consider the mathematical
pendulum, treating both the pendulum length and the macro time-step size as
variables in the surrogate. The Hamiltonian is
\[
\Ham(q,p;l)
=
\frac{p^2}{2ml^2}+mgl(1-\cos q),
\qquad
m=1,\quad g=9.81,\quad l\in\mathcal P.
\]
We fix the final time \(T=20.0\) and consider the parameter domain 
$\mathcal P=[0.7,0.8]$. As parameter in the surrogate, we use $\mu:=l.$
In contrast to \Cref{sec:wave}, we train a single surrogate on the full
augmented domain involving the mixed state, the length parameter, and the macro
time-step size.

The training parameter values and macro time-step sizes are
\[
\mu_{\mathrm{train}}
=
\{0.7,\ 0.725,\ 0.75,\ 0.775,\ 0.8\},
\qquad
\mathcal T_{\mathrm{train}}
=
\{0.04,\ 0.05,\ 0.06\}.
\]
For every \(\mu\in \mu_{\mathrm{train}}\) and every
\(\Delta T\in\mathcal T_{\mathrm{train}}\), initial states
\(x_0=(q_0,p_0)\) are sampled from an energy-bounded subset of the phase space,
following the sampling strategy of \cite{herkert2026kernel}. To this end, we
define the box
\[
\hat D
:=
[-\pi,\pi]\times\left[-2\sqrt{g},2\sqrt{g}\right],
\]
and the energy-bounded domain
\[
D
:=
\left\{(q,p)\in\hat D:\ \mathcal H(q,p)<2g\right\}.
\]
In practice, we generate a uniform tensor-product grid \(X\) on
\(\hat D\) with \(200\times 200\) grid points and retain only the
energy-admissible points. The resulting set of sampled initial states is
denoted by \(X_{0,\mathrm{train}}\).

We consider two sampling scenarios. In the first, we define the full training index set
$D_{\mathrm{train}}
:=
\mu_{\mathrm{train}}
\times
\mathcal T_{\mathrm{train}}
\times
X_{0,\mathrm{train}}.$
Using all sampled energy-admissible initial states gives the training set
\[
\mathcal M_{\mathrm A}
:=
\left\{
\left(
\bigl(I_1x_0+I_2\Phi_\mu^{\Delta T}(x_0),\,\mu,\,\Delta T\bigr),
J_2^\top
\frac{\Phi_\mu^{\Delta T}(x_0)-x_0}{\Delta T}
\right)
:
(\mu,\Delta T,x_0)\in D_{\mathrm{train}}
\right\}.
\]
In the second scenario, we restrict the sampled initial states to $D^-
:=
\{(q,p)\in\Omega:\ p\le 0\}.$
In particular, we use the restricted training index set
$
D_{\mathrm{train}}^-
:=
D_{\mathrm{train}}
\cap
\bigl(
\mu_{\mathrm{train}}
\times
\mathcal T_{\mathrm{train}}
\times
D^-
\bigr).$
This yields
\[
\mathcal M_{\mathrm B}
:=
\left\{
\left(
\bigl(I_1x_0+I_2\Phi_\mu^{\Delta T}(x_0),\,\mu,\,\Delta T\bigr),
J_2^\top
\frac{\Phi_\mu^{\Delta T}(x_0)-x_0}{\Delta T}
\right)
:
(\mu,\Delta T,x_0)\in D_{\mathrm{train}}^-
\right\}.
\]
Testing is performed on five random initial conditions with
\[
q_0\sim\mathcal U([0,\pi/2]),
\qquad
p_0=0,
\]
and on five additional parameter values drawn uniformly from \([0.7,0.8]\).
Each trajectory is evolved up to \(T=20.0\). We compare
(i) the proposed symplectic kernel predictor,
(ii) the implicit midpoint rule applied directly with macro step \(\Delta T\),
and
(iii) the micro-step reference solution.
For every considered value of \(\Delta T\), the relative errors are averaged
over the sampled initial conditions and test parameter values.

As a first result, \Cref{fig:pendulum} summarizes both the interpolation and prediction performance of the fully parameter-augmented surrogate. The left panel shows the decay of the $f$-greedy training and validation residuals as functions of the number of selected centers, while the right panel compares the average relative trajectory error of the symplectic predictor with the implicit-midpoint baseline for three randomly chosen macro time-step sizes, namely \(\Delta T\in \mathcal{T}_\textrm{test}:=
\{0.0509, 0.04847, 0.05292\}\). \begin{figure}[h]
 \centering
 \begin{subfigure}[t]{0.6\textwidth}
  \centering
\includegraphics[width=\linewidth]{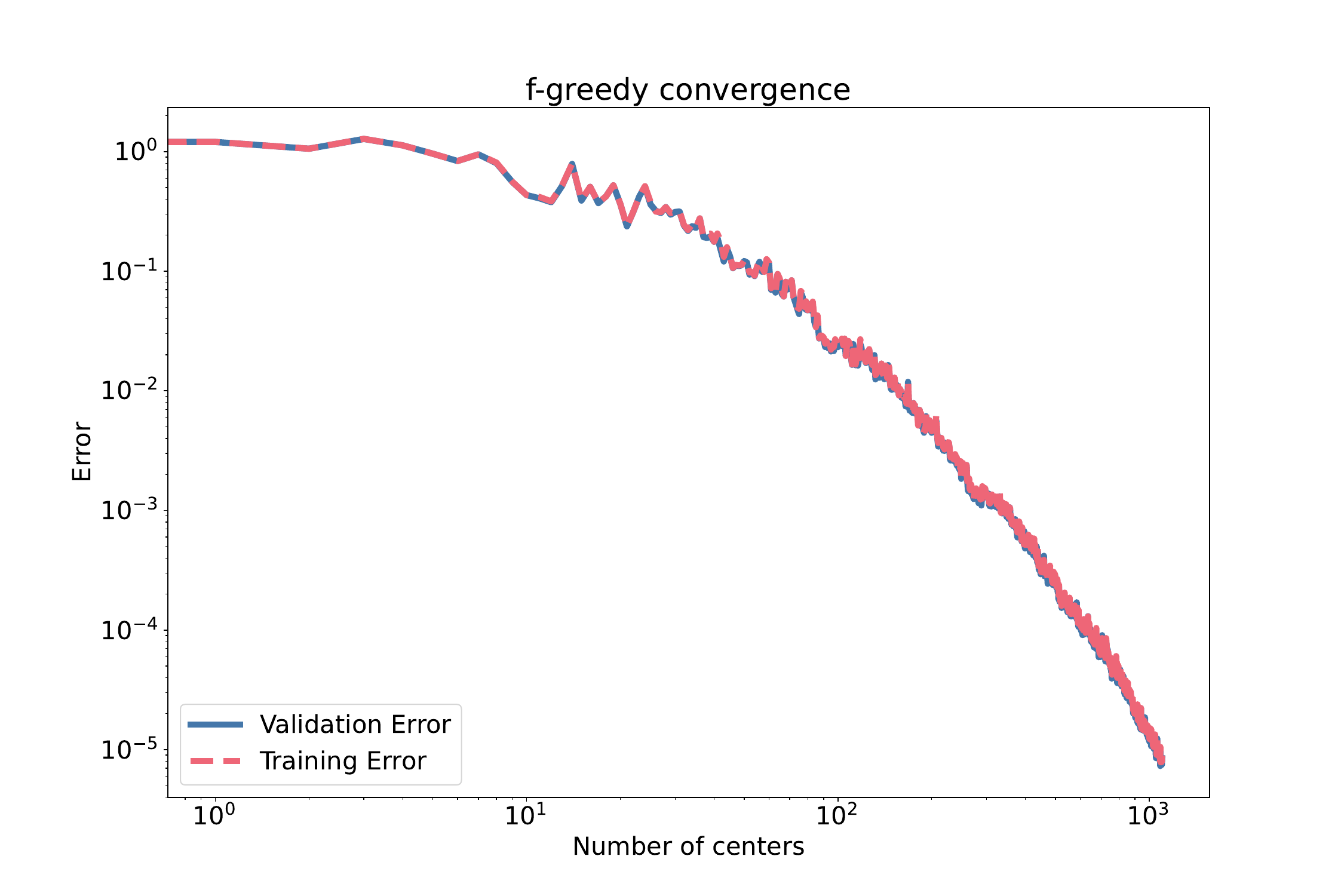}
  \caption{$f$-greedy interpolation error versus the number of selected centers for the full parameter-augmented surrogate, showing training (solid) and validation (dashed) curves.}
  \label{fig:fgreedy-convergence-param-time}
 \end{subfigure}\hfill
 \begin{subfigure}[t]{0.39\textwidth}
  \centering
  \includegraphics[width=\linewidth]{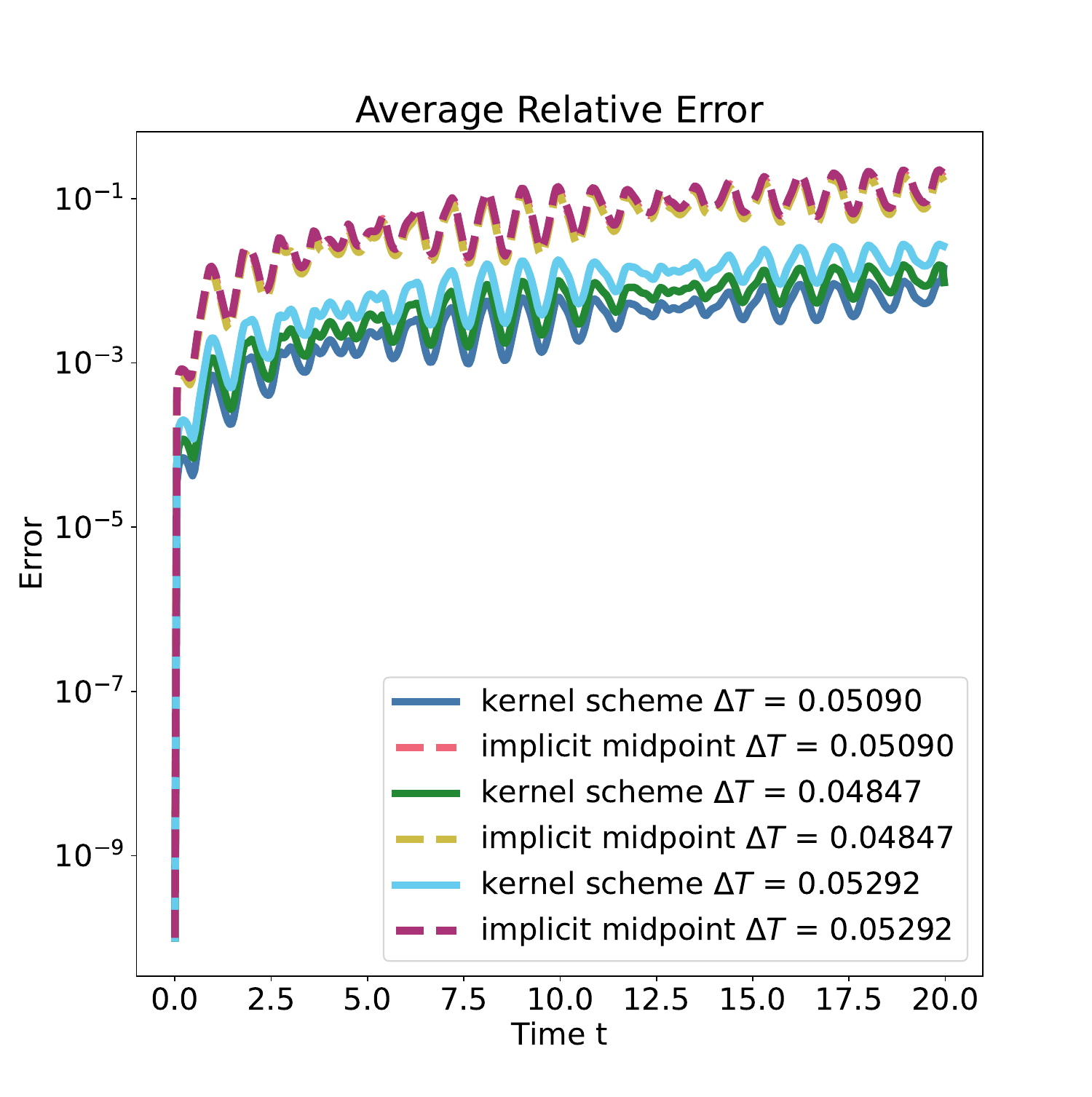}
  \caption{Relative error over time comparing the symplectic kernel scheme (solid) with the implicit-midpoint baseline (dashed) for three representative macro time-step sizes.}
  \label{fig:rel-reduction-error-param-time}
 \end{subfigure}
 \caption{Pendulum scenario $\mathcal{A}$: (a) \(f\)-greedy residual over the
number of centers; (b) average relative error over time.}\label{fig:pendulum}
\end{figure} The residual curves in \Cref{fig:fgreedy-convergence-param-time} exhibit a short pre-asymptotic regime for very small numbers of selected centers, followed by a steady decay over approximately five orders of magnitude. Starting from values of order \(10^{0}\), both the training and validation errors remain near that level up to roughly \(m\approx 5\)–\(10\), after which they enter an almost linear regime on the log--log scale. By the largest numbers of selected centers, the errors are reduced to about \(10^{-5}\). Throughout the entire range of \(m\), the two curves are almost indistinguishable, which again indicates stable interpolation behavior and no visible overfitting in the joint state--parameter--time setting. The corresponding trajectory errors are shown in \Cref{fig:rel-reduction-error-param-time}. Over the full time interval \([0,20]\), all three kernel-scheme curves remain bounded and oscillatory, with errors ranging roughly from \(10^{-5}\) up to a few \(10^{-3}\). The implicit-midpoint curves display the same qualitative oscillatory behavior, but at substantially larger error levels, typically between \(10^{-3}\) and \(10^{-1}\). Thus, for each of the displayed macro time-step sizes, the learned predictor improves upon the structure-preserving baseline by about two orders of magnitude. Overall, the figure shows that, within the training domain, the full product-kernel surrogate retains the bounded long-time behavior characteristic of symplectic approximations while achieving clearly smaller trajectory errors than direct macro-step integration.

We next turn to a more demanding generalization test. In contrast to the previous experiment, where both the pendulum length and the macro time-step remain inside the training domain, we now reduce the amount of training data and evaluate the learned predictor on parameter instances outside the sampled range. This allows us to assess how robustly the full parameter-augmented surrogate extrapolates in the length parameter.
In
\Cref{fig:rel-reduction-error-param-time-gen} the predictors are
evaluated in the extrapolatory regime. As a non-structure-preserving baseline, we compare to a direct kernel model which learns the coarse flow map itself, i.e.,
\[
s_{\mathrm{dir}}:\ (x_k,\mu,\Delta T)\mapsto \widehat x_{k+1}\approx \Phi_\mu^{\Delta T}(x_k),
\]
without embedding the prediction into an implicit symplectic update.\begin{figure}[h]
 \centering
 \begin{subfigure}[t]{0.32\textwidth}
  \centering
  \includegraphics[width = \textwidth]{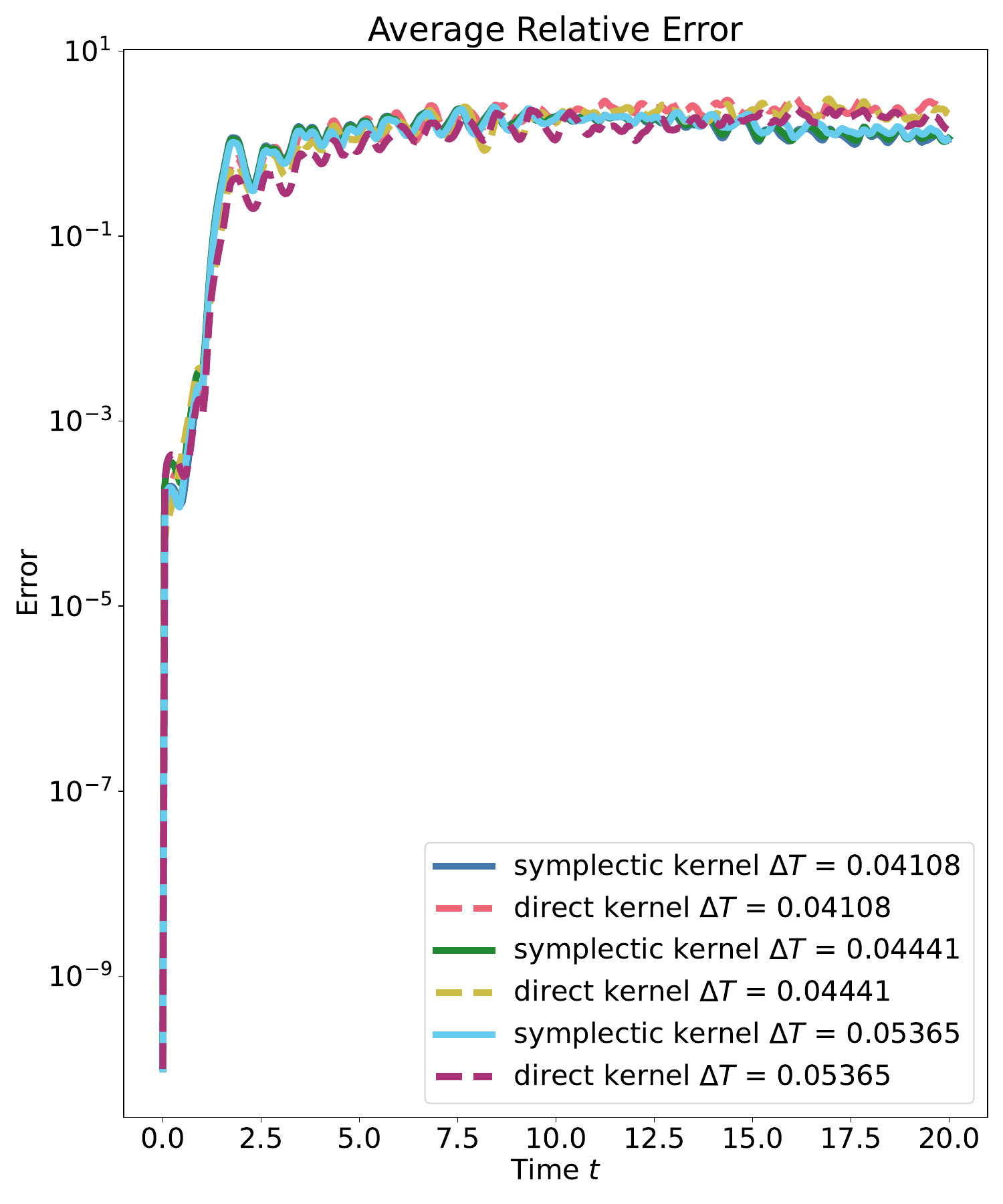}
  \caption{Relative error over time comparing the symplectic kernel scheme (solid) with the direct kernel baseline (dashed) in the generalization test.}
  \label{fig:rel-reduction-error-param-time-gen}
 \end{subfigure}\hfill
 \begin{subfigure}[t]{0.67\textwidth}
   \centering
  \includegraphics[width=\linewidth]{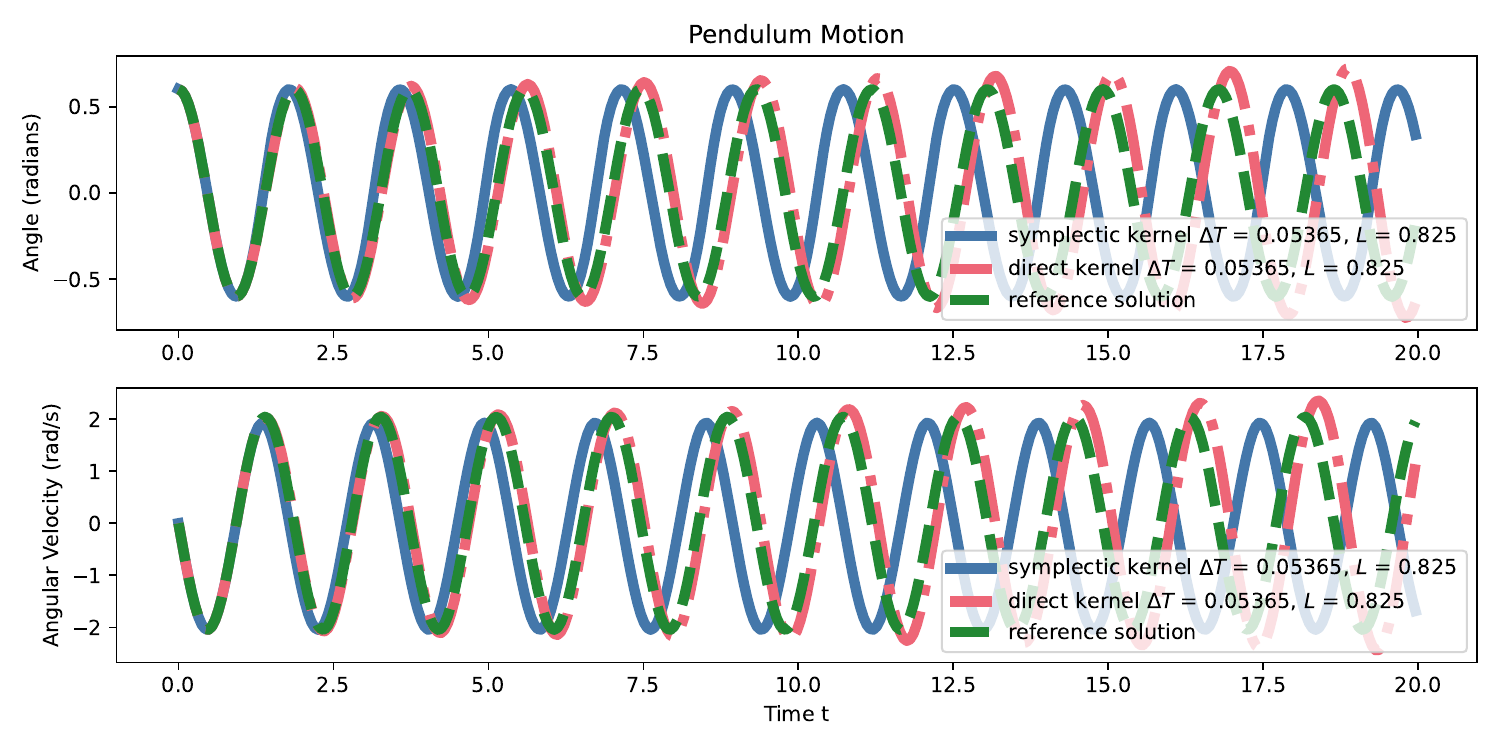}
  \caption{Pendulum parameter--time generalization test: comparison of one
  extrapolatory trajectory produced by the fully parameter-augmented
  symplectic kernel predictor and the direct kernel model with the
  reference solution for \(\Delta T=0.05365\) and \(\mu=0.825\).}
  \label{fig:pendulum-param-time-gen-trajectories}
 \end{subfigure}
\caption{Pendulum scenario $\mathcal{B}$: (a) average relative error over time;
(b) comparison of trajectories.}\label{fig:pendulum_param-time-gen}
\end{figure}
For all three representative macro time
step sizes, the errors of the symplectic kernel predictor start from very small
values near \(t=0\) and then increase before settling into bounded oscillatory
regimes where the direct kernel model and the symplectic predictor attain similar relative trajectory error. To illustrate the extrapolatory regime in more detail,
\Cref{fig:pendulum-param-time-gen-trajectories} compares one trajectory produced
by the fully parameter-augmented symplectic kernel predictor and the standard
direct kernel model with the corresponding reference solution. The test case is
given by \(\Delta T=0.05365\) and \(\mu=0.825\), and is therefore outside the
sampled range in the parameter variables. The symplectic kernel predictor reproduces the oscillatory motion with almost
constant amplitude over the full time interval. Although a phase shift relative
to the reference solution is visible, the extrema of both the angle \(q(t)\) and
the angular momentum \(p(t)\) remain at approximately the correct magnitude.
Thus, the structure-preserving predictor does not introduce a systematic change
of the energy level and retains the qualitative mechanics of the conservative
pendulum system. The direct kernel model behaves differently. Its oscillation amplitude grows over time, most clearly in
the angular momentum component. This amplitude growth is unphysical for the
conservative pendulum system and indicates that the direct model does not
preserve the underlying structure. Hence, the comparison shows a
clear qualitative distinction: the main error of the symplectic predictor is a
phase error, whereas the direct kernel model additionally introduces an
unphysical amplitude drift.

\begin{figure}[h]
 \centering
 \begin{adjustbox}{valign=t}
 \begin{minipage}{0.55\textwidth}
  \centering
  \includegraphics[width=\linewidth]{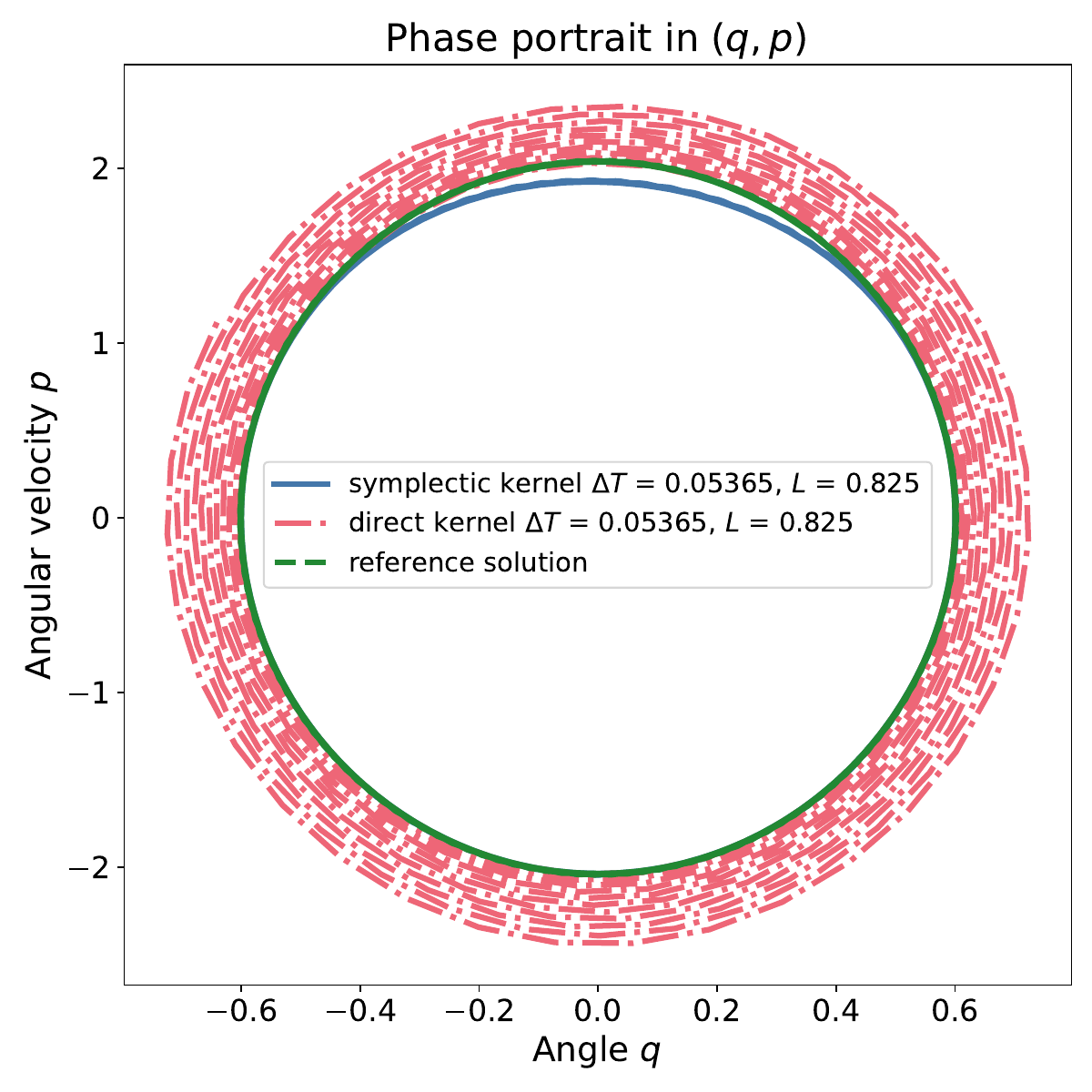}
  \subcaption{Phase portrait in the \((q,p)\)-plane.}
  \label{fig:pendulum-param-time-gen-phase}
 \end{minipage}
 \end{adjustbox}
 \hfill
 \begin{adjustbox}{valign=t}
 \begin{minipage}{0.43\textwidth}
  \centering
  \includegraphics[width=\linewidth]{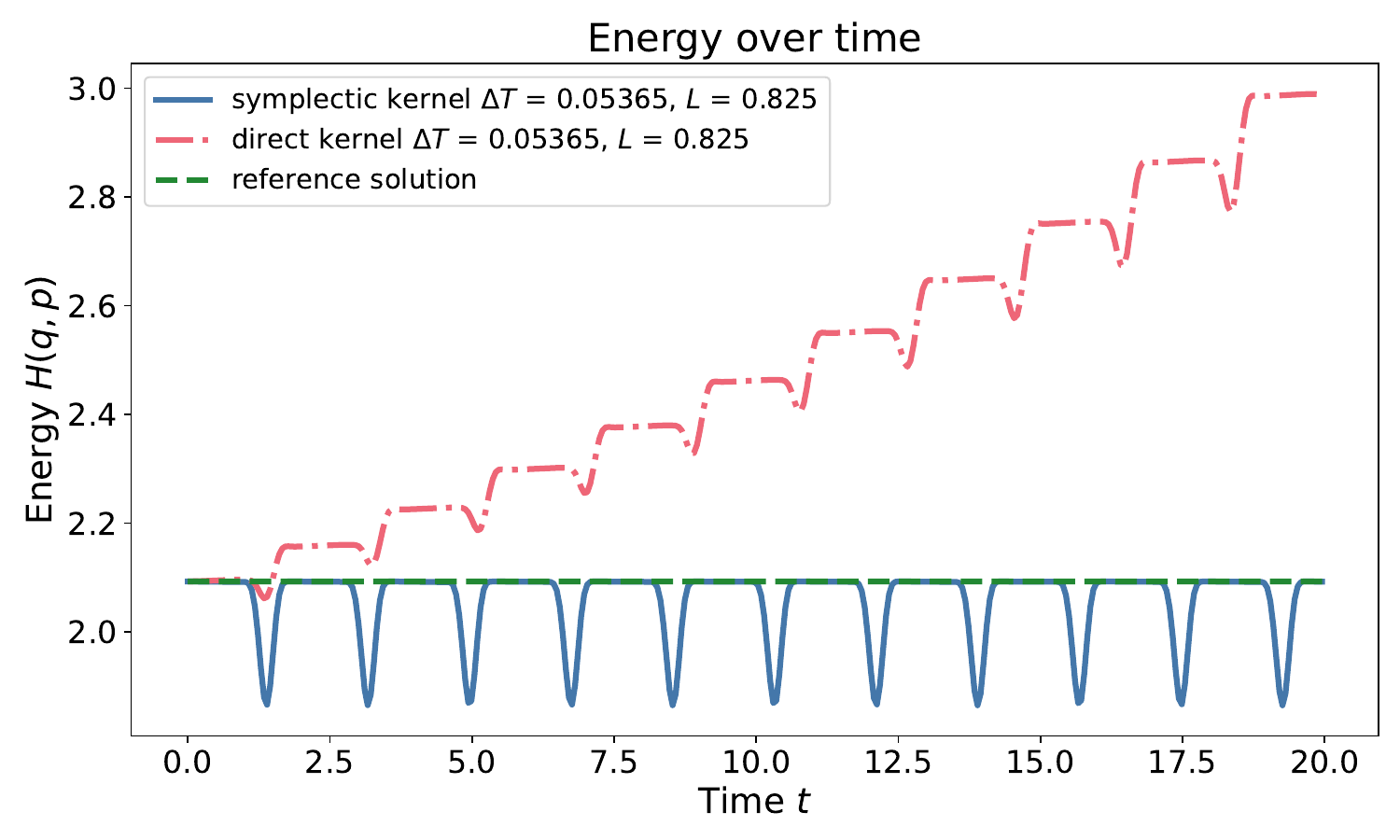}
  \subcaption{Hamiltonian along the trajectory.}
  \label{fig:pendulum-param-time-gen-hamiltonian}

  \vspace{0.8em}

  \includegraphics[width=\linewidth]{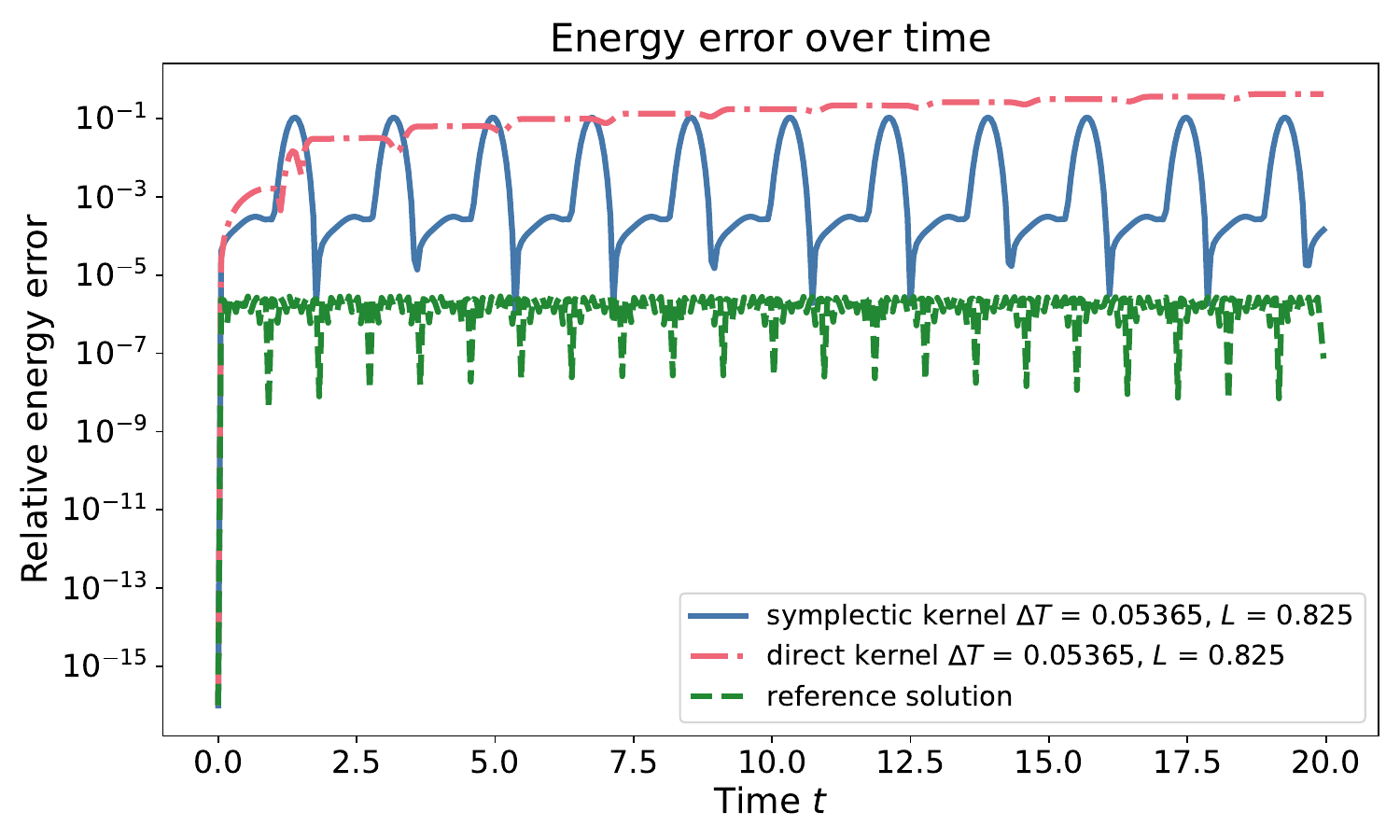}
  \subcaption{Relative Hamiltonian error over time.}
  \label{fig:pendulum-param-time-gen-energy-error}
 \end{minipage}
 \end{adjustbox}

\caption{Pendulum scenario \(\mathcal B\) for \(\mu=0.825\) and
\(\Delta T=0.05365\): (a) phase portrait, (b) Hamiltonian evolution, and
(c) relative Hamiltonian error for the fully parameter-augmented symplectic
kernel predictor, the direct kernel model, and the micro-step reference
solution.}
 \label{fig:pendulum-param-time-gen-energy-phase}
\end{figure}

The qualitative observations from the trajectory comparison are further confirmed by \Cref{fig:pendulum-param-time-gen-energy-phase}, which shows the phase portrait, Hamiltonian evolution, and relative Hamiltonian error for the same extrapolatory test case. In particular, the symplectic kernel predictor remains close to the reference orbit and exhibits no systematic Hamiltonian drift. Instead, its Hamiltonian shows small bounded oscillations around the reference value. This is consistent with symplectic approximations, which generally preserve a modified Hamiltonian rather than the original one exactly. By contrast, the direct kernel model drifts outward in the phase space, corresponding to an artificial increase of the Hamiltonian. The energy plots confirm this behavior: the direct model gains energy in a step-like manner and produces substantially larger Hamiltonian errors. This nonphysical energy growth indicates that the direct kernel surrogate does not preserve the underlying Hamiltonian structure. The micro-step reference solution exhibits only a very small relative Hamiltonian error, confirming the accuracy of the highly resolved structure-preserving reference integration. Hence, for a length parameter outside the training interval, the symplectic
kernel construction preserves the qualitative mechanics of the pendulum
considerably more reliably than the direct kernel model.

\section{Conclusion and Outlook}\label{Sec:Conclusion}

In this work, we extended the kernel-based symplectic predictor of \cite{herkert2026kernel} to a parameter-augmented setting for autonomous parameter-dependent Hamiltonian systems. The core idea is to learn a scalar surrogate on an augmented domain of mixed state variables, physical parameters, and macro time-step variables, and to insert its state gradient into an implicit symplectic-Euler-type update. This yields a family of large-step prediction maps that are symplectic by construction for every admissible parameter and time-step instance. In this way, the proposed approach combines direct flow-map learning over large macro time steps with preservation of the canonical symplectic structure.

From the analytical point of view, we formulated training as gradient HB interpolation in an RKHS on the augmented domain. Parameter and time-step dependence are incorporated through a product-kernel construction, while differentiation is taken only with respect to the state variables. Sparse surrogates are obtained by greedy center selection. Moreover, we showed that the convergence analysis from the non-augmented setting carries over to the product-kernel framework and yields quantitative bounds for both the state-gradient approximation and the resulting prediction error.

From the algorithmic point of view, we combined the proposed predictor with structure-preserving model reduction in order to treat larger systems efficiently. In particular, the reduced wave-equation example illustrates that symplectic model reduction and the parameter-augmented kernel predictor interact naturally: the reduced model remains Hamiltonian, the learned large-step surrogate preserves symplecticity, and both offline training and online evaluation become feasible in low dimension. The numerical experiments confirm the effectiveness of the proposed framework. The parameter-dependent reduced wave equation exhibits stable error decay and high accuracy across the tested parameter range. In particular, the product-kernel surrogate significantly reduces the trajectory error compared with direct macro-step integration. The pendulum examples show similar behavior: the greedy procedure yields surrogates with stable error decay, and within the interpolation regime the learned predictor clearly outperforms the structure-preserving implicit-midpoint baseline at the same macro step size.

Several directions for future work appear particularly promising. On the algorithmic side, it would be natural to generalize the present construction from the symplectic-Euler-type update to higher-order symplectic Runge--Kutta methods. This could improve accuracy for large macro steps while retaining the structure-preserving character of the learned predictor. On the theoretical side, it would be desirable to go beyond the native-space setting and develop approximation results for symplectic maps or generating functions outside the RKHS induced by the chosen kernel. Such a theory would help clarify the expressive power of kernel-based symplectic predictors and their approximation capabilities beyond the current interpolation framework. Furthermore, a natural next step would be to extend the existence results for the target function, namely the type~II generating function, to the augmented setting. Beyond these two directions, it also seems worthwhile to apply the present framework to broader classes of large-scale nonlinear Hamiltonian PDE models.

\section*{Aacknowledgements}
Funded by Deutsche Forschungsgemeinschaft (DFG,
German Research Foundation) under Germany´s
Excellence Strategy – EXC 2075/2 – 390740016.

\end{document}